\documentclass[a4]{amsart}

\input xypic 
\input xy 
\xyoption{all}

\usepackage[T1]{fontenc}

\usepackage{amssymb} 
\usepackage{bbm}
\usepackage{tikz-cd}
\usepackage[all]{xy}    
\usepackage{graphicx}
\usepackage{mathrsfs}
\usepackage{multirow}
\usepackage{float}
\usepackage{adjustbox}
\usepackage{amsthm}
\usepackage{accents}
\usepackage{mathtools}
\usepackage{tikz}
\usepackage{caption}

\usepackage[bookmarksnumbered, bookmarksopen,
colorlinks,citecolor=blue,linkcolor=blue,backref]{hyperref}

\newtheorem{theorem}{Theorem}[section]
\newtheorem{proposition}[theorem]{Proposition}
\newtheorem{lemma}[theorem]{Lemma}
\newtheorem{corollary}[theorem]{Corollary}

\theoremstyle{definition}

\newtheorem{remark}[theorem]{Remark}

\newcounter{bean}

\newcommand{\qqed}{\hfill\Box}

\begin{document}


\title[Homotopy of $6$-manifolds]
   {Suspension homotopy of $6$-manifolds} 

\author{Ruizhi Huang} 
\address{Institute of Mathematics, Academy of Mathematics and Systems Science, 
   Chinese Academy of Sciences, Beijing 100190, China} 
\email{haungrz@amss.ac.cn} 
   \urladdr{https://sites.google.com/site/hrzsea/}

\subjclass[2020]{Primary 
55P15, 
55P40, 
57R19; 
Secondary 
55P10, 
55N15
}
\keywords{$6$-manifolds, suspension, homotopy decomposition, Poincar\'{e} duality space, $K$-groups}


\begin{abstract} 
For a simply connected closed orientable manifold of dimension $6$, we show its homotopy decomposition after double suspension. This allows us to determine its $K$- and $KO$-groups easily. Moreover, for a special case we refine the decomposition to show the rigidity property of the manifold after double suspension. 
\end{abstract}

\maketitle

\tableofcontents

\section{Introduction} 
Let $M$ be a closed orientable smooth manifold of dimension $n$. There are tremendous investigations in geometric topology to classify the diffeomorphism or homeomorphism type of $M$ in various cases. For instance, in the general case, Wall \cite{Wal1, Wal3} studied the $(s-1)$-connected $2s$-manifolds and the $(s-1)$-connected $(2s+1)$-manifolds. For the concrete case with specified dimension $n$, Bardon \cite{Bar} classified the simply connected $5$-manifolds, and Wall \cite{Wal2}, Jupp \cite{Jup} and Zhubr \cite{Zhu1, Zhu2} classified the simply connected $6$-manifolds. More recently, Kreck and Su \cite{KS} classified certain non-simply connected $5$-manifolds, while Crowley and Nordstr\"{o}m \cite{CN} and Kreck \cite{Kre} studied the classification problem of various kinds of $7$-manifolds.

In the mentioned literature, the homotopy classification of $M$ was usually carried out as a byproduct in terms of a system of invariants. However, it is almost impossible to extract nontrivial homotopy information of $M$ directly from the classification. On the other way around, the unstable homotopy theory is a powerful tool to study the homotopy properties of manifolds. There are several interesting attempts in recent years along this way. For instance, Beben and Theriualt \cite{BT1} studied the loop decompositions of $(s-1)$-connected $2s$-manifolds, while Beben and Wu \cite{BW} and Huang and Theriault \cite{HT} studied the loop decompositions of the $(s-1)$-connected $(2s+1)$-manifolds. The homotopy groups of these manifolds were also investigated by Sa. Basu and So. Basu \cite{BB, Bas} from different point of view. Moreover, a theoretical method of loop decomposition was developed by Beben and Theriault \cite{BT2}, which is quite useful for studying the homotopy of manifolds. Additionally, the homotopy type of the suspension of a connected $4$-manifold was determined by So and Theriault \cite{ST}.

In this paper, we study the homotopy aspect of simply connected $6$-manifolds. Let $M$ be a simply connected closed orientable $6$-manifold. By Poincar\'{e} duality and the universal coefficient theorem we have 
\begin{equation}\label{HMeqintro}
H_\ast(M; \mathbb{Z})=\left\{\begin{array}{ll}
\mathbb{Z}^{\oplus d}\oplus T &\ast=2,\\
\mathbb{Z}^{\oplus 2m}\oplus T & \ast=3,\\
\mathbb{Z}^{\oplus d} & \ast=4,\\
\mathbb{Z}  & \ast=0, 6\\
0&\hbox{otherwise},
\end{array}
\right.
\end{equation}
where $m$, $d\geq 0$, and $T$ is a finitely generated abelian torsion group. 
Our first main theorem concerns the double suspension splitting of $M$. 
Denote $\Sigma X$ be the suspension of any $CW$-complex $X$.
Denote $P^n(T)$ be the Moore space such that the reduced cohomology $\widetilde{H}^\ast(P^n(T);\mathbb{Z})\cong T$ if $\ast=n$ and $0$ otherwise \cite{N2}.
\begin{theorem}\label{M6thm}
Let $M$ be a simply connected closed orientable $6$-manifold with homology of the form (\ref{HMeqintro}).
Suppose that $T$ has no $2$ or $3$-torsion. Then there is an integer $c$ with $0\leq c\leq d$ determined by the cohomology ring of $M$ such that
\begin{itemize}
\item if $c=0$, 
\[
\Sigma^2 M \simeq \Sigma W_0\vee  \bigvee_{j=1}^{d-1}(S^{4}\vee S^6)\vee  \bigvee_{i=1}^{2m}S^5 \vee P^6(T)\vee P^5(T),
\]
where $W_0\simeq (S^{3}\vee S^5)\cup e^7$;
\item if $c=d$, 
\[
\Sigma^2 M \simeq \Sigma W_d\vee \bigvee_{i=1}^{d-1}\Sigma^2 \mathbb{C}P^2\vee  \bigvee_{i=1}^{2m}S^5 \vee P^6(T)\vee P^5(T),
\]
where $W_d\simeq \Sigma \mathbb{C}P^2\cup e^7$;
\item if $1\leq c\leq d-1$,
\[
\Sigma^2 M \simeq \Sigma W_c\vee \bigvee_{i=1}^{c-1}\Sigma^2 \mathbb{C}P^2 \vee \bigvee_{j=1}^{d-c-1}(S^{4}\vee S^6)\vee  \bigvee_{i=1}^{2m}S^5 \vee P^6(T)\vee P^5(T),
\]
where $W_c\simeq (\Sigma \mathbb{C}P^2 \vee S^{3}\vee S^5)\cup e^7$.
\end{itemize}
\end{theorem}
Notice the number $c$ is indeed determined by the Steenrod square $Sq^2: H^2(M;\mathbb{Z}/2\mathbb{Z})\rightarrow H^4(M;\mathbb{Z}/2\mathbb{Z})$, while there is an ambiguous term $W_c$ ($0\leq c\leq d$). 
Since we only need the suspension of $W_c$ and the Hopf element $\eta_i\in \pi_{i+1}(S^i)$ is detected by $Sq^2$, the ambiguity reduces to the components of the attaching map of the top cell of $W_c$ to the wedge summand $\Sigma \mathbb{C}P^2$ and $S^3$. 

Nevertheless, Theorem \ref{M6thm} is still useful, for instance, to calculate the $K$-group or the $KO$-group of $M$ in Corollary \ref{KM6thm}. In particular, when $M$ is a Calabi-Yau threefold it partially reproduces the result of Doran and Morgan \cite[Corollary 1.10]{DM} on its $K$-group by different method, and provides new computation on its $KO$-group. Moreover, there are many examples of simply connected Calabi-Yau threefolds. For instance, based on Kreuzer and Skarke \cite{KSk}, Batyrev and Kreuzer \cite{BK} showed that there are exactly 473 800 760 families of simply connected Calabi-Yau $3$-folds corresponding to $4$-dimensional reflexive polytopes.

\begin{corollary}\label{KM6thm}
Let $M$ be the manifold in Theorem \ref{M6thm}. Then for the reduced $K$-group and $KO$-groups of $M$ 
\[
\widetilde{K}(M)\cong \mathbb{Z}^{\oplus 2d+1} \oplus T, \ \ \ ~\widetilde{KO}(M)\cong \mathop{\bigoplus}\limits_{d}(\mathbb{Z}\oplus \mathbb{Z}/2\mathbb{Z}).
\]
\end{corollary}

If we specify to the case when $d=1$, we can obtain the complete description of $M$ after double suspension, based on the work of Yamaguchi \cite{Yam} (also summarised and corrected by Baues \cite{Bau}). Denote $\eta_i^3=\eta_{i+2}\circ \eta_{i+1}\circ \eta_i\in \pi_{i+3}(S^i)$ \cite{Tod}, where $\eta_i\in \pi_{i+1}(S^i)$ is the Hopf element. Let $V_3$ be the manifold as the total space of the sphere bundle of the oriented $\mathbb{R}^3$-bundle over $S^4$ determined by its first Pontryagin class $p_1=12s_4$, where $s_4\in H^4(S^4;\mathbb{Z})$ is a generator.

\begin{theorem}\label{M6d=1thm}
Let $M$ be a simply connected closed orientable $6$-manifold with homology of the form (\ref{HMeqintro}) such that $d=1$. 
Let $x$, $y\in H^\ast(M;\mathbb{Z})$ be two generators such that ${\rm deg}(x)=2$ and ${\rm deg}(y)=4$. Denote $x^2=k y$ for some $k\in \mathbb{Z}$.
Suppose $T$ has no $2$ or $3$-torsion. Then
\begin{itemize}
\item if $k$ is odd, then $M$ is Spin, moreover 
\[
\Sigma^2 M\simeq \Sigma^2 \mathbb{C}P^3\vee  \bigvee_{i=1}^{2m}S^5\vee P^6(T)\vee P^5(T),
\]
when $k \equiv 1 ~{\rm mod}~6$, while
\[
 \Sigma^2 M\simeq \Sigma^2 V_3\vee  \bigvee_{i=1}^{2m}S^5\vee P^6(T)\vee P^5(T),
\]
when $k \equiv 3 ~{\rm mod}~6$;
\item if $k$ is even and $V$ is non-Spin
\[
 \Sigma^2 M\simeq S^4\vee \Sigma^4 \mathbb{C}P^2\vee  \bigvee_{i=1}^{2m}S^5\vee P^6(T)\vee P^5(T);
\]
\item if $k$ is even and $V$ is Spin
\[
\Sigma^2 M\simeq (S^4\cup_{\lambda\eta_4^{3}}e^8)\vee S^6\vee  \bigvee_{i=1}^{2m}S^5\vee P^6(T)\vee P^5(T),
\]
where $\lambda\in\mathbb{Z}/2$ is determined by $M$.
\end{itemize}
\end{theorem}
It should be remarked that there is no ambiguity in the term $(S^4\cup_{\lambda\eta_4^{3}}e^8)$ in the last decomposition of $\Sigma^2 M$. Indeed, the stable cube element $\eta^3_n\in \pi_{n+3}(S^n)$ ($n\geq 2$) is detected by the secondary operation $\mathbb{T}$ \cite[Exercise 4.2.5]{Har}, and in our case the homotopy decomposition has to preserve the module structure induced by the cohomology operations. Moreover, it is clear that the number $k~{\rm mod}~2$ and the spin condition of $M$ are determined by the Steenrod square $Sq^2$. And we will also see that $\Sigma \mathbb{C}P^3$ and $\Sigma V_3$ can be distinguished by the Steenrod power $\mathcal{P}^1: H^3(\Sigma M;\mathbb{Z}/3\mathbb{Z})\rightarrow H^7(\Sigma M;\mathbb{Z}/3\mathbb{Z})$.
Hence, we obtain the following rigidity result for manifolds in Theorem \ref{M6d=1thm} after double suspension.
\begin{corollary}\label{rigthm}
Let $M$ and $M^\prime$ be two manifolds in Theorem \ref{M6d=1thm}. Then $\Sigma^2 M\simeq \Sigma^2 M^\prime$ if and only if $H^\ast(\Sigma^2 M;\mathbb{Z})\cong H^\ast(\Sigma^2 M^\prime;\mathbb{Z})$ as abelian groups, and $H^\ast(\Sigma^2 M;\mathbb{Z}/p\mathbb{Z})\cong H^\ast(\Sigma^2 M^\prime;\mathbb{Z}/p\mathbb{Z})$ as $\mathbb{Z}/2\mathbb{Z}\{Sq^2, \mathbb{T}\}$-modules when $p=2$, and as $\mathbb{Z}/3\mathbb{Z}\{\mathcal{P}^1\}$-modules when $p=3$.
\end{corollary}

The paper is organized as follows. In Section \ref{sec: redb3=0} we reduce the decomposition problem of $6$-manifolds to that of ones whose third Betti numbers are zero. In Section \ref{sec: hdec}, we give a detailed procedure to decompose $6$-manifolds after double suspension by homology decomposition method. Section \ref{sec: pfthm} and Section \ref{sec: pfthmd=1} are devoted to prove Theorem \ref{M6thm} and Theorem \ref{M6d=1thm} respectively. In Section \ref{sec: hgps}, we compute some homotopy groups of odd primary Moore spaces used in Section \ref{sec: hdec}.

\bigskip

\noindent{\bf Acknowledgements.}
Ruizhi Huang was supported by National Natural Science Foundation of China (Grant nos. 11801544 and 11688101), and ``Chen Jingrun'' Future Star Program of AMSS. He would like to thank Professor Stephen Theriault for his international online lecture series ``Loop Space Decomposition'', which stimulated his research interest in the homotopy of $6$-manifolds.

\section{Reducing to the case when $b_3(M)=0$} 
\label{sec: redb3=0} 
The following well known splitting theorem of $6$-manifolds was proved by Wall \cite{Wal2} in smooth category, while Jupp \cite{Jup} pointed out that the theorem holds in topological category by the same argument.
\begin{theorem}\cite[Theorem 1]{Wal2}\label{wallsplitthm}
Let $M$ be a simply connected closed orientable $6$-manifold with third Betti number $b_3(M)=2m$. Then there exists a $6$-manifold $M_1$ such that
\[
 \hspace{6cm}
M\cong M_1 \# \mathop{\#}\limits_{m} S^3\times S^3.
 \hspace{6cm}\Box
\]
\end{theorem}

\begin{corollary}\label{splitcor}
Let $M$ and $M_1$ be the manifolds in Theorem \ref{wallsplitthm}. Then 
\[
\Sigma M\simeq \Sigma M_1\vee \bigvee_{i=1}^{m}(S^4\vee S^4).
\]
\end{corollary}
\begin{proof}
Let $M_1^\prime$ and $M^\prime$ be the $5$-skeletons of $M_1$ and $M$ respectively. It is known that $S^3\vee S^3$ is the $5$-skeleton of $S^3\times S^3$ and $\Sigma(S^3\times S^3)\simeq \Sigma (S^3\vee S^3)\vee S^7$. In particular, there is a homotopy retraction $r: \Sigma(S^3\times S^3)\rightarrow \Sigma (S^3\vee S^3)$.
For the connected sum $M_1\# (S^3\times S^3)$, there are the obvious pinch maps $q_1: M_1\# (S^3\times S^3)\rightarrow M_1$ and $q_2: M_1\# (S^3\times S^3)\rightarrow S^3\times S^3$.
Consider the composition
\[
\phi: \Sigma (M_1\# (S^3\times S^3))\stackrel{\mu^\prime}{\longrightarrow}  \Sigma (M_1\# (S^3\times S^3)) \vee \Sigma (M_1\# (S^3\times S^3))\stackrel{E q_1\vee (r\circ E q_2)}{\longrightarrow} \Sigma M_1 \vee \Sigma (S^3\vee S^3),
\]
where $\mu^\prime$ is the standard co-multiplication of suspension complex, and $E$ denotes the suspension of a map. It is easy to see that $\phi$ induces an isomorphism on homology and then is a homotopy equivalence by the Whitehead Theorem.
Since
\[
M^\prime \simeq M_1^\prime \vee \bigvee_{i=1}^{m}(S^3\vee S^3),
\]
by repeating the above argument, we can show the decomposition in the corollary.
\end{proof}

In Theorem \ref{wallsplitthm}, the connected summand $M_1$ satisfies $b_3(M_1)=0$. Hence by Corollary \ref{splitcor} it suffices to consider such $6$-manifolds in the sequel.

\section{Homology decomposition of $M$ after suitable suspensions}
\label{sec: hdec}
Let $M$ be a simply connected closed orientable $6$-manifold with $b_3(M)=0$. By Poincar\'{e} duality and the universal coefficient theorem we have 
\begin{equation}\label{HMeq}
H_\ast(M; \mathbb{Z})=\left\{\begin{array}{ll}
\mathbb{Z}^{\oplus d}\oplus T &\ast=2,\\
T & \ast=3,\\
\mathbb{Z}^{\oplus d} & \ast=4,\\
\mathbb{Z}  & \ast=0, 6\\
0&\hbox{otherwise},
\end{array}
\right.
\end{equation}
where $d\geq 0$, and $T$ is a finitely generated abelian torsion group. We may denote 
\begin{equation}\label{Teq}
T=\mathop{\oplus}\limits_{k=1}^{\ell}\mathbb{Z}/p_k^{r_k}\mathbb{Z},
\end{equation}
where each $p_k$ is a prime and $r_k\geq 1$.

Instead of using skeleton decomposition, we may apply {\it homology decomposition} to study the cell structure of $M$. For any finitely generated abelian group $A$, let $P^n(A)$ be the Moore space such that the reduced cohomology $\widetilde{H}^\ast(P^n(A);\mathbb{Z})\cong A$ if $\ast=n$ and $0$ otherwise \cite{N2}. The information on the homotopy groups of $P^n(T)$ used in this section will be proved in Section \ref{sec: hgps}.

\begin{theorem}\cite[Theorem 4H.3]{Hat} \label{hdecthm}
Let $X$ be a simply connected $CW$-complex. Denote $H_i=H_i(X;\mathbb{Z})$. Then there is a sequence of complex $X_{(i)}$ ($i\geq 2$) such that 
\begin{itemize}
\item $H_j(X_{(i)};\mathbb{Z})\cong H_{j}(X;\mathbb{Z})$ for $j\leq i$ and $H_j(X_{(i)};\mathbb{Z})=0$ for $j>i$;
\item $X_{(2)}=P^{3}(H_2)$, and $X_{(i)}$ is defined by a homotopy cofibration
\[
 P^{i}(H_i)\stackrel{f_{i-1}}{\longrightarrow} X_{(i-1)}\stackrel{
 \iota_{i-1}}{\longrightarrow} X_{(i)},
\]
where $f_{i-1}$ induces a trivial homomorphism $f_{i-1\ast}: H_{i-1}(P^{i}(H_i);\mathbb{Z})\rightarrow H_{i-1}(X_{(i-1)};\mathbb{Z})$;
\item $X\simeq \mathop{{\rm hocolim}}\limits\{X_{(2)}\stackrel{
 \iota_{2}}{\longrightarrow}\cdots \stackrel{\iota_{i-2}}{\longrightarrow}X_{(i-1)}\stackrel{ \iota_{i-1}}{\longrightarrow} X_{(i)}\stackrel{\iota_{i}}{\longrightarrow}\cdots\}$. ~$\qqed$
\end{itemize}
\end{theorem}
From this theorem, it is clear that the homology decomposition is compatible with the suspension functor. That is, for $X$ in Theorem \ref{hdecthm} the sequence of the triples $(\Sigma X_{(i)}, E f_{i}, E \iota_i)$ contributes to a homology decomposition of $\Sigma X$.

\subsection{Structure of $\mathbf{M_{(3)}}$}\label{subsec: M3}
By Theorem \ref{hdecthm} and (\ref{HMeq}),
\begin{equation}\label{M2eq}
\begin{split}
M_{(2)}&\simeq\mathop{\bigvee}\limits_{i=1}^{d}S^2\vee P^3(T),\\
P^{3}(T)&\stackrel{f_{2}}{\longrightarrow} M_{(2)}\stackrel{
 \iota_{2}}{\longrightarrow} M_{(3)}.
\end{split}
\end{equation}
Notice that from (\ref{Teq}), we have $P^n(T)\simeq \mathop{\bigvee}\limits_{k=1}^{\ell}P^n(p_{k}^{r_k})$ by \cite{N3} or \cite{N2}.
\begin{lemma}\label{M3lemma}
The map $f_{2}$ in (\ref{M2eq}) is null-homotopic, and hence
\[
M_{(3)}\simeq\mathop{\bigvee}\limits_{i=1}^{d}S^2\vee P^3(T)\vee P^4(T).
\]
\end{lemma}
\begin{proof}
Since $P^3(T)\simeq \mathop{\bigvee}\limits_{k=1}^{\ell}P^3(p_{k}^{r_k})$, there is the embedding
$j: \mathop{\bigvee}\limits_{i=1}^{\ell}S^2\rightarrow P^3(T)$ of the bottom cells.
Consider the following commutative diagram
\[
 \xymatrix{
 \pi_2(\mathop{\bigvee}\limits_{i=1}^{\ell}S^2) \ar[d]^{{\rm hur}}_{\cong} \ar@{->>}[r]^{j_\ast} &
 \pi_2(P^3(T)) \ar[r]^{f_{2\ast}}  \ar[d]^{{\rm hur}}_{\cong} &
 \pi_2(M_{(2)}) \ar[d]^{{\rm hur}}_{\cong} \\
 H_2(\mathop{\bigvee}\limits_{i=1}^{\ell}S^2;\mathbb{Z}) \ar@{->>}[r]^{j_\ast} &
  H_2(P^3(T);\mathbb{Z}) \ar[r]^{f_{2\ast}=0}& 
  H_2(M_{(2)};\mathbb{Z}),
  }
\]
where the Hurewicz homomorphisms ${\rm hur}$ are isomorphisms by the Hurewicz theorem, $f_{2\ast}=0$ on homology by Theorem \ref{hdecthm}, and both $j_\ast$ are epimorphisms.
In particular, $f_{2\ast}\circ j_\ast$ is trivial on homotopy groups, and hence $f_2\circ j$ is null homotopic. 
Then with (\ref{M2eq}) we have the diagram of homotopy cofibrations
\[
 \xymatrix{
 \mathop{\bigvee}\limits_{i=1}^{\ell}S^2 \ar[d]^{\mathop{\bigvee}\limits_{i=1}^{\ell}p_k^{r_k}}\ar[r] & 
\ast \ar[r] \ar[d]&
 \mathop{\bigvee}\limits_{i=1}^{\ell}S^3 \ar[d]^{i_2\circ(\mathop{\bigvee}\limits_{i=1}^{\ell}p_k^{r_k})}\\
 \mathop{\bigvee}\limits_{i=1}^{\ell}S^2 \ar[r]^{0}  \ar[d]^{j} &
 M_{(2)}\ar@{=}[d] \ar[r]&
M_{(2)}\vee \mathop{\bigvee}\limits_{i=1}^{\ell}S^3 \ar[d] \\
P^3(T) \ar[r]^{f_2} &
M_{(2)} \ar[r]^{\iota_2} &
M_{(3)},
   }
\]
where $p_k^{r_k}: S^{n}\rightarrow S^{n}$ is a map of degree $p_k^{r_k}$, and $i_2: \mathop{\bigvee}\limits_{i=1}^{\ell}S^3\rightarrow M_{(2)}\vee\mathop{\bigvee}\limits_{i=1}^{\ell}S^3$ is the injection onto the sphere summands. It follows that
\[
M_{(3)}\simeq M_{(2)}\vee P^4(T)\simeq \mathop{\bigvee}\limits_{i=1}^{d}S^2\vee P^3(T)\vee P^4(T),
\]
and the proof of the lemma is completed.
\end{proof}
The following corollary follows from Lemma \ref{M3lemma} and will be used in Lemma \ref{M61lemma}.
\begin{corollary}\label{Vdeflemma}
The homotopy cofiber of the obvious inclusion $j: P^3(T)\vee P^4(T)\stackrel{}{\rightarrow}M_{(3)}\rightarrow M$ is a Poincar\'{e} duality complex $V$ with cell structure
\[
\hspace{4.8cm} 
V= \bigvee_{i=1}^{d} S^{2} \cup e^{4}_{(1)}\cup e^{4}_{(2)}\ldots \cup e^{4}_{(d)}\cup e^{6}.  
\hspace{4.8cm}\Box
\]
\end{corollary}
Moreover, by \cite[Theorem 8]{Wal2} $V$ is homotopy equivalent to a closed smooth manifold.

\subsection{Structure of $\mathbf{M_{(5)}}$}\label{subsec: M51}
By Theorem \ref{hdecthm} and Lemma \ref{M3lemma},
\begin{equation}\label{M3eq}
\begin{split}
M_{(3)}&\simeq\mathop{\bigvee}\limits_{i=1}^{d}S^2\vee P^3(T)\vee P^4(T),\\
\mathop{\bigvee}\limits_{i=1}^{d}S^3&\stackrel{f_{3}}{\longrightarrow} M_{(3)}\stackrel{
 \iota_{3}}{\longrightarrow} M_{(4)}=M_{(5)}.
\end{split}
\end{equation}
We need to study the map
\[
f_3: \mathop{\bigvee}\limits_{i=1}^{d}S^3\rightarrow \mathop{\bigvee}\limits_{i=1}^{d}S^2\vee P^3(T)\vee P^4(T).
\]
Let $i_3: P^4(T)\rightarrow \mathop{\bigvee}\limits_{i=1}^{d}S^2\vee P^3(T)\vee P^4(T)$ be the inclusion. 
Define the complex $Y$ by the homotopy cofibration
\[
P^4(T)\stackrel{\iota_{3} \circ i_3}{\longrightarrow} M_{(4)}=M_{(5)}\longrightarrow Y.
\]

\begin{lemma}\label{M5lemma1}
The map $f_3$ in (\ref{M3eq}) factors as
\[
f_3:  \mathop{\bigvee}\limits_{i=1}^{d}S^3\stackrel{f_3^\prime}{\longrightarrow} \mathop{\bigvee}\limits_{i=1}^{d}S^2\vee P^3(T) \stackrel{i_1\vee i_2}{\longrightarrow} \mathop{\bigvee}\limits_{i=1}^{d}S^2\vee P^3(T)\vee P^4(T),
\]
for some $f_3^\prime$, where $i_1$ and $i_2$ are inclusions. Moreover, there is the homotopy cofibration
\begin{equation}\label{Yeq}
\mathop{\bigvee}\limits_{i=1}^{d}S^3\stackrel{f_3^\prime}{\longrightarrow} \mathop{\bigvee}\limits_{i=1}^{d}S^2\vee P^3(T) \stackrel{\iota_3^\prime}{\longrightarrow} Y,
\end{equation}
and
\[
M_{(5)}\simeq Y\vee P^4(T).
\]
\end{lemma}
\begin{proof}
First there is the diagram of homotopy cofibrations
\[
 \xymatrix{
\ast \ar[r] \ar[d]&
P^4(T) \ar@{=}[r] \ar[d]^{i_3} &
P^4(T) \ar[d]^{\iota_3\circ i_3} \\
\mathop{\bigvee}\limits_{i=1}^{d}S^3\ar@{=}[d] \ar[r]^{f_3} &
M_{(3)} \ar[d]^{q_{1,2}} \ar[r]^{\iota_3} &
M_{(5)} \ar[d]\\
\mathop{\bigvee}\limits_{i=1}^{d}S^3 \ar[r]^>>>>{f_3^\prime} &
\mathop{\bigvee}\limits_{i=1}^{d}S^2\vee P^3(T) \ar[r]^<<<<{\iota_3^\prime} &
Y,
}
\]
where $q_{1,2}$ is the obvious projection, $\iota_3^\prime$ is induced from $\iota_3$, and $f_3^\prime:=q_{1,2}\circ f_3$. The diagram immediately implies that (\ref{Yeq}) is a homotopy cofibration. 

Denote $q_3: \mathop{\bigvee}\limits_{i=1}^{d}S^2\vee P^3(T)\vee P^4(T)\rightarrow P^4(T)$ be the canonical projection.
By Theorem \ref{hdecthm}, $f_{3\ast}: H_3(\mathop{\bigvee}\limits_{i=1}^{d}S^3;\mathbb{Z})\rightarrow H_3(\mathop{\bigvee}\limits_{i=1}^{d}S^2\vee P^3(T)\vee P^4(T);\mathbb{Z})$ is trivial. In particular, $q_{3\ast}\circ f_{3\ast}=0$. Then by the Hurewicz Theorem $q_3\circ f_3$ is null homotopic.
Further, by the Hilton-Milnor Theorem (see \S XI.6 of \cite{Whi} for instance), $
\pi_3(S^2\vee P^3(T)\vee P^4(T))\cong \pi_3(S^2\vee P^3(T))\oplus \pi_3(P^4(T))$, and hence
\[
[\mathop{\bigvee}\limits_{i=1}^{d}S^3, S^2\vee P^3(T)\vee P^4(T)]\cong [\mathop{\bigvee}\limits_{i=1}^{d}S^3, S^2\vee P^3(T)]\oplus [\mathop{\bigvee}\limits_{i=1}^{d}S^3, P^4(T)].
\]
Under this isomorphism, the homotopy class of $f_3$ corresponds to $[f_3^\prime]+[q_3\circ f_3]$. However since we already show that $[q_3\circ f_3]=0$, we have $f_3\simeq (i_1\vee i_2)\circ f_3^\prime$, and $M_{(5)}\simeq Y\vee P^4(T)$ as required.
\end{proof}

\subsection{Structure of $\mathbf{\Sigma M_{(5)}}$}\label{subsec: M52}
From this point, we may need extra conditions on the torsion group $T$. First recall that we already showed that by Lemma \ref{M5lemma1}
\begin{equation}\label{M52eq}
\begin{split}
M_{(5)}&\simeq Y\vee P^4(T),\\
\mathop{\bigvee}\limits_{i=1}^{d}S^3&\stackrel{f_3^\prime}{\longrightarrow} \mathop{\bigvee}\limits_{i=1}^{d}S^2\vee P^3(T) \stackrel{\iota_3^\prime}{\longrightarrow} Y.
\end{split}
\end{equation}
Let $q_1: \mathop{\bigvee}\limits_{i=1}^{d}S^2\vee P^3(T)\rightarrow \mathop{\bigvee}\limits_{i=1}^{d}S^2$ be the canonical projection. Denote $f_3^{\prime\prime}:=q_1\circ f_3^\prime=q_1\circ q_{1,2}\circ f_3$.
Also recall $P^n(T)\simeq \mathop{\bigvee}\limits_{k=1}^{\ell}P^n(p_{k}^{r_k})$. Let us suppose each $p\geq 3$ from now on. 
\begin{lemma}\label{M5lemma2}
Suppose $T$ has no $2$-torsion.
\[
\Sigma M_{(5)}\simeq \Sigma X\vee P^5(T)\vee P^4(T),
\]
where $X$ is the homotopy cofiber of the map $f_3^{\prime\prime}: \mathop{\bigvee}\limits_{i=1}^{d}S^3\rightarrow \mathop{\bigvee}\limits_{i=1}^{d}S^2$.
\end{lemma}
\begin{proof}
There is the diagram of homotopy cofibrations
\[
 \xymatrix{
\ast \ar[r] \ar[d]&
P^3(T) \ar@{=}[r] \ar[d]^{i_2} &
P^3(T) \ar[d]^{\iota_3^\prime\circ i_2} \\
\mathop{\bigvee}\limits_{i=1}^{d}S^3\ar@{=}[d] \ar[r]^<<<{f_3^\prime} &
\mathop{\bigvee}\limits_{i=1}^{d}S^2\vee P^3(T) \ar[d]^{q_{1}} \ar[r]^>>>>>{\iota_3^\prime} &
Y \ar[d]\\
\mathop{\bigvee}\limits_{i=1}^{d}S^3 \ar[r]^<<<<<{f_3^{\prime\prime}} &
\mathop{\bigvee}\limits_{i=1}^{d}S^2 \ar[r] &
X,
}
\]
where $i_2$ is the canonical inclusion. Since $\pi_4(P^4(p^r))=0$ for odd $p$ by Lemma \ref{pinpnlemma}, 
we have $\Sigma Y\simeq \Sigma X \vee P^4(T)$. The lemma then follows from (\ref{M52eq}).
\end{proof}

\subsection{Structure of $\mathbf{\Sigma^2 M}$}\label{subsec: M61}
Recall we have when $T$ has no $2$ torsion by Lemma \ref{M5lemma2}
\begin{equation}\label{M61eq}
\begin{split}
\Sigma M_{(5)}&\simeq \Sigma X\vee P^5(T)\vee P^4(T),\\
S^5&\stackrel{f_5}{\longrightarrow} M_{(5)}\stackrel{\iota_5}{\longrightarrow} M.
\end{split}
\end{equation}
Further by Corollary \ref{Vdeflemma} we have the homotopy cofibration
\[
S^5\rightarrow X\rightarrow V,
\]
where $X$ is defined in Lemma \ref{M5lemma2} without restriction on $T$, and $V$ is a closed smooth manifold.
We may further suppose $T$ has no $3$-torsion. 
\begin{lemma}\label{M61lemma}
Suppose $T$ has no $2$ or $3$-torsion. Then
\[
\Sigma^2 M\simeq \Sigma^2 V\vee P^6(T)\vee P^5(T).
\]
\end{lemma}
\begin{proof}
By the Hilton-Milnor theorem, we may write the suspension of $f_5$ as 
\[
E f_5=g_5^{(1)}+g_5^{(2)}+g_5^{(3)}+\theta: S^6\rightarrow  \Sigma M_{(5)}\simeq \Sigma X\vee P^5(T)\vee P^4(T) ,
\]
for some $\theta$, where $E \theta=0$, $g_5^{(i)}=q_i\circ E f_5$, and $q_i$ is the canonical projection of $ \Sigma X\vee P^5(T)\vee P^4(T)$ onto its $i$-th summand. Then $g_5^{(2)}=0$ by Lemma \ref{pin+1pnlemma}, and $E g_5^{(3)}=0$ by Lemma \ref{pi6p4pnlemma}. It follows that $E^2 f_5=E g_5^{(1)}$. Furthermore, there is the diagram of homotopy cofibrations
\[
\xymatrix{
\ast \ar[r] \ar[d]&
P^4(T)\vee P^3(T) \ar@{=}[r] \ar[d]^{j_5} &
P^4(T)\vee P^3(T) \ar[d]^{j}\\
S^5 \ar[r]^{f_5} \ar@{=}[d] &
M_{(5)} \ar[r]^{\iota_5} \ar[d]^{\pi_5}&
M\ar[d]^{\pi}\\
S^5\ar[r] &
X\ar[r]&
V,
}
\]
where the homotopy cofibration in the last column is defined in Corollary \ref{Vdeflemma} by using Lemma \ref{M3lemma}, and similarly the homotopy cofibration in the middle column can be also defined by using Lemma \ref{M3lemma}. Then it is clear $g_5^{(1)}\simeq E(\pi_5\circ f_5)$ and the lemma follows.
\end{proof}

\section{Proof of Theorem \ref{M6thm} and Corollary \ref{KM6thm}}
\label{sec: pfthm}
In Lemma \ref{M61lemma} we have established the double suspension splitting of $M$ when $b_3(M)=0$, and are left to consider the homotopy of $V$ after suspension. Recall that $V$ is a Poincar\'{e} Duality complex of dimension $6$, and its $5$-skeleton $V_5=X$ is the homotopy cofiber of the map $f_3^{\prime\prime}: \mathop{\bigvee}\limits_{i=1}^{d}S^3\rightarrow \mathop{\bigvee}\limits_{i=1}^{d}S^2$ by Lemma \ref{M5lemma2}. The following lemma, as a special case of \cite[Lemma 6.1]{H}, determines the suspension homotopy type of $X$.
\begin{lemma}\cite[Lemma 6.1]{H}\label{Xdeclemma}
\[
\Sigma X\simeq \bigvee_{i=1}^{c}\Sigma \mathbb{C}P^2 \vee \bigvee_{j=1}^{d-c}(S^{3}\vee S^5),
\]
for some $0\leq c\leq d$. ~$\qqed$
\end{lemma}
We may apply the method in \cite[Section 3]{H} to decompose $\Sigma^2 V$, in the same way that we used it to prove \cite[Lemma 6.4 and Lemma 6.6]{H}.
\begin{lemma}\label{Vdeclemma}
Suppose $\Sigma X$ decomposes as in Lemma \ref{Xdeclemma}.
\begin{itemize}
\item If $c=0$, 
\[
\Sigma^2 V \simeq \Sigma W_0\vee  \bigvee_{j=1}^{d-1}(S^{4}\vee S^6),
\]
where $W_0\simeq (S^{3}\vee S^5)\cup e^7$;
\item if $c=d$, 
\[
\Sigma^2 V \simeq \Sigma W_d\vee \bigvee_{i=1}^{d-1}\Sigma^2 \mathbb{C}P^2,
\]
where $W_d\simeq \Sigma \mathbb{C}P^2\cup e^7$;
\item if $1\leq c\leq d-1$,
\[
\Sigma^2 V \simeq \Sigma W_c\vee \bigvee_{i=1}^{c-1}\Sigma^2 \mathbb{C}P^2 \vee \bigvee_{j=1}^{d-c-1}(S^{4}\vee S^6),
\]
where $W_c\simeq (\Sigma \mathbb{C}P^2 \vee S^{3}\vee S^5)\cup e^7$.
\end{itemize}
\end{lemma} 
\begin{proof}
As we pointed out that the proof is similar to that of \cite[Lemma 6.4 and Lemma 6.6]{H}, we may only sketch it. The interested reader can find the details of the method in \cite[Section 3]{H}.
With Lemma \ref{Xdeclemma} let $g: S^6 \rightarrow \Sigma X\simeq \bigvee_{i=1}^{c}\Sigma \mathbb{C}P^2 \vee \bigvee_{j=1}^{d-c}(S^{3}\vee S^5)$ be the attaching map of the top cell of $\Sigma V$.
To apply the method in \cite[Section 3]{H}, we only need the information of homotopy groups $\pi_{6}(\Sigma \mathbb{C}P^2)\cong \mathbb{Z}/6\mathbb{Z}$ by \cite[Proposition 8.2(i)]{Muk}, $\pi_6(S^3)\cong \mathbb{Z}/12\mathbb{Z}$ and $\pi_6(S^5)\cong \mathbb{Z}/2$, which are all finite cyclic groups.
Then we can represent the attaching map $E g$ of the top cell of $\Sigma^2 V$ by a matrix $B$, and apply \cite[Lemma 3.1]{H} to transform $B$ to a simpler one $C$. The new matrix representation $C$ of the attaching map, corresponding to a base change of $\Sigma X$ through a self homotopy equivalence, will give the desired decomposition.
\end{proof}

Now we can prove Theorem~\ref{M6thm}. 

\begin{proof}[Proof of Theorem~\ref{M6thm}] 
First by Theorem \ref{wallsplitthm} and Corollary \ref{splitcor}, we have
\[
\Sigma M\simeq \Sigma M_1\vee \bigvee_{i=1}^{m}(S^4\vee S^4),
\]
where $M_1$ is a closed $6$-manifold with homology of the form (\ref{HMeq}). In particular $b_3(M_1)=0$.
Hence by Lemma \ref{M61lemma} and Lemma \ref{Vdeclemma}, we have that if $1\leq c\leq d-1$
\[
\Sigma^2 M\simeq \Sigma W_c\vee \bigvee_{i=1}^{c-1}\Sigma^2 \mathbb{C}P^2 \vee \bigvee_{j=1}^{d-c-1}(S^{4}\vee S^6)\vee P^6(T)\vee P^5(T)\vee \bigvee_{i=1}^{2m}S^5,
\]
where $W_c\simeq (\Sigma \mathbb{C}P^2 \vee S^{3}\vee S^5)\cup e^7$. The decompositions for the other two cases when $c=0$ or $d$ can be obtained similarly. Finally, notice that $c$ records the number of the non-trivial Steenrod square $Sq^2: H^2(\Sigma^2 M;\mathbb{Z}/2\mathbb{Z})\rightarrow H^4(\Sigma^2 M;\mathbb{Z}/2\mathbb{Z})$, which is preserved by the decomposition and the suspension operator. 
Since $Sq^2$ is the cup square on the elements of $H^2(M;\mathbb{Z}/2\mathbb{Z})$, this completes the proof of Theorem \ref{M6thm}.
\end{proof} 

To prove Corollary \ref{KM6thm}, we need the Bott periodicity showed in the following table:
\begin{table}[H]
\centering
\captionof{table}{$\widetilde{K}^{-i}(S^0)$ and $\widetilde{KO}^{-j}(S^0)$}
\begin{tabular}{c | c | c }
  $i~{\rm mod}~2$ &  $0$ & $1$  \\ \hline
$\widetilde{K}^{-i}(S^0)$ &$\mathbb{Z}$ &$0$ 
\end{tabular}
\ \ \ 
\begin{tabular}{c | c | c | c | c | c | c | c | c }
 $j~{\rm mod}~8$ &  $0$ & $1$ & $2$ & $3$ & $4$ & $5$ & $6$ & $7$ \\ \hline
$\widetilde{KO}^{-j}(S^0)$ & $\mathbb{Z}$ & $\mathbb{Z}/2\mathbb{Z}$ &$\mathbb{Z}/2\mathbb{Z}$ & $0$ &$\mathbb{Z}$ & $0$ & $0$ &$0$ 
\end{tabular}
\label{Bott}
\end{table}
From Table (\ref{Bott}) we can easily calculate the following, where only $\widetilde{KO}^2(P^5(T))=0$ requires that $T$ has no $2$-torsion.
\begin{lemma}\label{Kgplemma}
Let $W_c$ be the complex defined in Lemma \ref{Vdeclemma} for $0\leq c\leq d$.
\begin{itemize}
\item
$
\widetilde{K}(P^5(T))\cong T, \ \ \ ~ \widetilde{K}(P^6(T))=0,
$
\item
$
\widetilde{KO}^2(P^5(T))=\widetilde{KO}^2(P^6(T))=0,
$
\item
$
\widetilde{KO}^1(\Sigma \mathbb{C}P^2)\cong
\widetilde{KO}^1(S^3\vee S^5)\cong
\mathbb{Z}\oplus \mathbb{Z}/2\mathbb{Z}, \ \ \ ~ 
\widetilde{KO}^1(\Sigma^2 \mathbb{C}P^2)\cong
\widetilde{KO}^1 (S^4\vee S^6)=0,
$
\item 
$
\widetilde{KO}^1(W_0)
\cong \widetilde{KO}^1(W_d)
\cong\mathbb{Z}\oplus \mathbb{Z}/2\mathbb{Z},  \ \ \ ~ 
\widetilde{KO}^1(W_c)
\cong\mathop{\bigoplus}\limits_{2}(\mathbb{Z}\oplus \mathbb{Z}/2\mathbb{Z}).
$  ~$\qqed$
\end{itemize}
\end{lemma}

\begin{proof}[Proof of Theorem~\ref{KM6thm}] 
Let us only compute the $\widetilde{KO}$-group of $M$ when $1\leq c\leq d-1$, while the other cases can be computed similarly. By Theorem \ref{M6thm} and Lemma \ref{Kgplemma}, we have 
\[
\begin{split}
\widetilde{KO}(M)\cong \widetilde{KO}^2(\Sigma^2 M)&\cong \widetilde{KO}^2(\Sigma W_c)\oplus \bigoplus_{i=1}^{c-1}\widetilde{KO}^2(\Sigma^2 \mathbb{C}P^2)\oplus \bigoplus_{j=1}^{d-c-1}\widetilde{KO}^2(S^{4}\vee S^6)\\
&\oplus \bigoplus_{j=1}^{2m}\widetilde{KO}^2(S^5)\oplus\widetilde{KO}^2(P^6(T))\oplus \widetilde{KO}^2(P^5(T))\\
&\cong \mathop{\bigoplus}\limits_{2}(\mathbb{Z}\oplus \mathbb{Z}/2\mathbb{Z})\oplus \mathop{\bigoplus}\limits_{c-1}(\mathbb{Z}\oplus \mathbb{Z}/2\mathbb{Z})\oplus \mathop{\bigoplus}\limits_{d-c-1}(\mathbb{Z}\oplus \mathbb{Z}/2\mathbb{Z})\\
&=\mathop{\bigoplus}\limits_{d}(\mathbb{Z}\oplus \mathbb{Z}/2\mathbb{Z}).
\end{split}
\]
\end{proof}

\section{The case when $H^2(M;\mathbb{Z})\cong \mathbb{Z}$}
\label{sec: pfthmd=1}
By Lemma \ref{M61lemma}, we may consider the torsion free case first.
In \cite{Yam}, Yamaguchi classified the homotopy types of $CW$-complexes of the form $V\simeq S^2\cup e^4\cup e^6$. Specifying to the case when $V$ is a manifold, we can summarized the necessary result in the following theorem (c.f. \cite[Section 1]{Bau}).
\begin{theorem}\cite[Corollary 4.6, Lemma 2.6, Lemma 4.3]{Yam}\label{Yamthm}
Let $V\simeq S^2\cup_{k\eta_2} e^4\cup_b e^6$ be a closed smooth manifold, where $k\eta_2$ with $k\in \mathbb{Z}$ and $b$ are the attaching maps of the cells $e^4$ and $e^6$ respectively.
Then the top attaching map $b$ is determined by a generator $b_W\in \pi_5(S^2\cup_{k\eta_2} e^4)$ of infinite order, the second Stiefel-Whitney class $\omega_2(V)\in H^2(V;\mathbb{Z}/2)$ of $V$, and an indeterminacy term $b^\prime\in\mathbb{Z}/2$ which depends on the following three cases.
\begin{itemize}
\item If $k$ is odd, then $V$ is Spin and the homotopy type of $V$ is unique determined by $k$, and $b=b_W$;
\item If $k$ is even and $V$ is non-Spin, then the homotopy type of $V$ is unique determined by $k$, and $b=b_W+\widetilde{\eta}_4$ with $\widetilde{\eta}_4$ representing the generator of a $\mathbb{Z}/2$ summand determined by $\omega_2(V)$;
\item If $k$ is even and $V$ is Spin, then $V$ has precisely two homotopy types depending on the value of $b^\prime\in\mathbb{Z}/2$, and $b=b_W+b^\prime$. ~$\qqed$ 
\end{itemize}
\end{theorem}
\begin{remark}\label{Yamremark}
In Theorem \ref{Yamthm}, $b_W$, as a generator of the $\mathbb{Z}$-summand of $\pi_5( S^2\cup_{k\eta_2} e^4)$, is indeed a relative Whitehead product when $k\neq 0$ by \cite[Lemma 2.6]{Yam}.
It is possible that the suspension map $E b_W$ is not null-homotopic.
 $\widetilde{\eta}_4$ is derived from the homotopy class of 
\[
S^5\stackrel{b}{\longrightarrow} S^2\cup_{k\eta_2} e^4\stackrel{q}{\longrightarrow} S^4,
\]
where $q$ is the quotient map onto the $4$-cell of $S^2\cup_{k\eta_2} e^4$ (c.f. \cite[Section 1]{Bau}). $b^\prime$ is from a class of $\pi_5(S^2)\cong \mathbb{Z}/2\{\eta_2^3\}$ by \cite[Lemma 2.6]{Yam} or \cite[Section 1]{Bau}. 
Also, as pointed out in Mathematical Reviews \cite{MR}, the original theorem of \cite{Yam} was misstated which is corrected here and in \cite[Section 1]{Bau} as well.
\end{remark}

Thanks to Theorem \ref{Yamthm} and Remark \ref{Yamremark}, we can describe the suspension homotopy type of $V$. Recall that $\pi_6(\Sigma \mathbb{C}P^2)\cong \mathbb{Z}/6\mathbb{Z}\{E \pi_2\}$ \cite[Theorem 8.2(i)]{Muk}, where $\pi_2: S^5\rightarrow \mathbb{C}P^2$ is the Hopf map, and $E$ is the suspension of a map.
\begin{proposition}\label{Vd=1decprop}
Let $V\simeq S^2\cup_{k\eta_2} e^4\cup_b e^6$ be a closed smooth manifold.
\begin{itemize}
\item If $k$ is odd, then $V$ is Spin and 
\[
\Sigma V\simeq \Sigma \mathbb{C}P^2\cup_{k^\prime E \pi_2} e^7,
\]
where $k^\prime=1$ or $3$ such that $k^\prime \equiv \pm k ~{\rm mod}~6$;
\item If $k$ is even and $V$ is non-Spin
\[
\Sigma V\simeq S^3\vee \Sigma^3 \mathbb{C}P^2;
\]
\item If $k$ is even and $V$ is Spin
\[
\Sigma V\simeq (S^3\cup_{b^\prime\eta_3^{3}}e^7)\vee S^5,
\]
where $b^\prime\in\mathbb{Z}/2$ is from Theorem \ref{Yamthm}. 
\end{itemize}
\end{proposition}
\begin{proof}
It is clear that the decompositions for the two cases when $k$ is even follows immediately from Theorem \ref{Yamthm} and Remark \ref{Yamremark}. When $k$ is odd, $V$ is spin and $\Sigma V\simeq \Sigma \mathbb{C}P^2\cup_{E b_W}e^7$ by Theorem \ref{Yamthm}. 
Also notice that $\Sigma \mathbb{C}P^2\cup_{E b_W}e^7\simeq \Sigma \mathbb{C}P^2\cup_{-E b_W}e^7$. Hence, to prove the statement in the proposition it suffices to show that the suspension map
\[
E: \pi_5(S^2\cup_{k\eta_2}e^4)\rightarrow \pi_6(\Sigma \mathbb{C}P^2)
\]
sends the generator $b_W$ to $kE\pi_2$ up to sign.

For this purpose, start with the diagram of homotopy cofibrations
\begin{equation}\label{rcp2diag}
\begin{gathered}
\xymatrix{
S^3 \ar[d]^{k} \ar[r]^{k\eta_2} & 
S^2 \ar@{=}[d] \ar[r] &
S^2\cup_{k\eta_2}e^4 \ar[d]^{r} \\
S^3 \ar[r]^{\eta_2} \ar[d] &
S^2 \ar[d] \ar[r] &
\mathbb{C}P^2 \ar[d]\\
P^4(k) \ar[r] &
\ast \ar[r] &
P^5(k),
}
\end{gathered}
\end{equation}
which defines the map $r$. Then there is the diagram of homotopy fibrations
\begin{equation}
\label{s1fibdiag}
\begin{gathered}
\xymatrix{
S^1\ar@{=}[d] \ar[r]&
Z\ar[r] \ar[d]^{\tilde{r}} & 
S^2\cup_{k\eta_2}e^4 \ar[r]^{f_x} \ar[d]^{r} &
K(\mathbb{Z},2) \ar@{=}[d] \\
S^1\ar[r]&
S^5\ar[r]^{\pi_2}  & 
\mathbb{C}P^2\ar[r]^<<<<<{f_c} &
K(\mathbb{Z},2), 
}
\end{gathered}
\end{equation}
where $f_c$ and $f_x$ represent the generators $c\in H^2(\mathbb{C}P^2;\mathbb{Z})$, and $x\in H^2(S^2\cup_{k\eta_2}e^4;\mathbb{Z})$ respectively, and $Z$ is the homotopy fibre of $f_x$ mapping to $S^5$ by the induced map $\tilde{r}$. By analyzing the Serre spectral sequences of the homotopy fibrations in Diagram (\ref{s1fibdiag}), it can be showed that $Z\simeq P^4(k)\cup e^5$ and $\tilde{r}^\ast: H^5(S^5;\mathbb{Z})\rightarrow H^5(Z;\mathbb{Z})$ is of degree $k$. Since by Lemma \ref{pinpnlemma} $\pi_4(P^4(k))=0$ when $k$ is odd, we see that $Z\simeq P^4(k)\vee S^5$, and then $\tilde{r}_\ast$ is of degree $k$ on homology. Moreover, by the naturality of the Hurewicz homomorphism and Lemma \ref{pin+1pnlemma}, it is easy to see that $\tilde{r}_\ast: \pi_5(Z)\cong \mathbb{Z}\rightarrow \pi_5(S^5)$ is of degree $k$.
It follows that $r_\ast:\pi_5(S^2\cup_{k\eta_2}e^4)\cong\mathbb{Z}\rightarrow \pi_5(\mathbb{C}P^2)\cong\mathbb{Z}$ is of degree $k$ by Diagram (\ref{s1fibdiag}).

Now the naturality of suspension map induces the following commutative digram
\begin{equation}\label{suscp2diag}
\begin{gathered}
\xymatrix{
\pi_5(S^2\cup_{k\eta_2}e^4) \ar[r]^<<<<{r_\ast} \ar[d]^{E}&
\pi_5(\mathbb{C}P^2) \ar[d]^{E} \\
\pi_6(\Sigma \mathbb{C}P^2) \ar[r]^{Er_\ast}&
\pi_6(\Sigma \mathbb{C}P^2),
}
\end{gathered}
\end{equation}
where $E: \pi_5(\mathbb{C}P^2)\cong\mathbb{Z}\rightarrow \pi_6(\Sigma \mathbb{C}P^2)\cong \mathbb{Z}/6\mathbb{Z}$ is surjective by \cite[Theorem 8.2(i)]{Muk}. We have showed that $r_\ast$ in Diagram (\ref{suscp2diag}) is of degree $k$.
On the other hand, from the last column of Diagram (\ref{rcp2diag}) we have the homotopy cofibraiton
\[
\Sigma \mathbb{C}P^2 \stackrel{Er}{\longrightarrow} \Sigma \mathbb{C}P^2 \stackrel{}{\longrightarrow} P^6(k).
\]
Applying the Blakers-Massey Theorem \cite{BM}, we obtain the exact sequence
\begin{equation}\label{selfkcp2eq}
\pi_6(\Sigma \mathbb{C}P^2) \stackrel{Er_\ast}{\longrightarrow}\pi_6(\Sigma \mathbb{C}P^2) \stackrel{}{\longrightarrow} \pi_6(P^6(k)).
\end{equation}
Since $\pi_6(P^6(k))=0$ by Lemma \ref{pinpnlemma} and $\pi_6(\Sigma \mathbb{C}P^2)\cong \mathbb{Z}/6\mathbb{Z}\{E\pi_2\}$, we see that $Er_\ast$ is an isomorphism from (\ref{selfkcp2eq}). Then by Diagram (\ref{suscp2diag}), it follows that $E: \pi_5(S^2\cup_{k\eta_2}e^4)\rightarrow \pi_6(\Sigma \mathbb{C}P^2)$ sends the generator $b_W$ to $kE\pi_2$ up to sign. This proves the statement in the case when $k$ is odd, and we have completed the proof of the proposition.
\end{proof}

Now we can prove Theorem~\ref{M6d=1thm} and Corollary \ref{rigthm}. 

\begin{proof}[Proof of Theorem~\ref{M6d=1thm}] 
First by Theorem \ref{wallsplitthm}, Corollary \ref{splitcor} and Lemma \ref{M61lemma}, we have
\[
\begin{split}
\Sigma^2 M
&\simeq \Sigma^2 M_1\vee \bigvee_{i=1}^{m}(S^5\vee S^5)\\
&\simeq \Sigma^2 V\vee P^6(T)\vee P^{5}(T)\vee \bigvee_{i=1}^{m}(S^5\vee S^5),
\end{split}
\]
where $M_1$ is a closed $6$-manifold with homology of the form (\ref{HMeq}) such that $b_3(M_1)=0$ and $d=1$. Moreover, By Corollary \ref{Vdeflemma} and the assumption on the ring structure of $H^\ast(M;\mathbb{Z})$, $V\simeq S^2\cup_{k\eta_2} e^4\cup_b e^6$ for some attaching map $b$. Denote $\lambda=b^\prime$ in Theorem \ref{Yamthm}. The theorem for the two cases when $k$ is even then follows immediately from Proposition \ref{Vd=1decprop}. For the case when $k$ is odd, recall that there is the fibre bundle \cite[Section 1.1]{HBJ}
\[
S^2\longrightarrow \mathbb{C}P^3\stackrel{\sigma}{\longrightarrow} S^4,
\]
with its first Pontryagin class $p_1=4s_4$ where $s_4\in H^4(S^4;\mathbb{Z})$ is a generator. Then pullback this bundle along the self-map of $S^4$ of degree $3$, we obtain the $6$-manifold $V_3$ with bundle projection $\sigma_3$ onto $S^4$ in the following diagram of $S^2$-bundles
\[
\xymatrix{
S^2\ar@{=}[d] \ar[r] &
V_3 \ar[r]^{\sigma_3} \ar[d] &
S^4\ar[d]^{3} \\
S^2\ar[r] &
\mathbb{C}P^3\ar[r]^{\sigma}  &
S^4.
}
\]
From this diagram, it is easy to see that the first Pontryagin class of $\sigma_3$ is $12s_4$ as required and $x^2=3y$, where by abuse of notation $x$, $y\in H^\ast(V_3;\mathbb{Z})$ are two generators such that ${\rm deg}(x)=2$ and ${\rm deg}(y)=4$. 
Hence by Proposition \ref{Vd=1decprop}, $\Sigma V\simeq \Sigma \mathbb{C}P^3$ when $k\equiv \pm 1~{\rm mod}~6$ and $\Sigma V\simeq \Sigma V_3$ when $k\equiv 3~{\rm mod}~6$, and then the two decompositions when $k$ is odd in the theorem follows. This completes the proof of the theorem.
\end{proof}

\bigskip

\begin{proof}[Proof of Corollary~\ref{rigthm}] 
As the discussions before Corollary~\ref{rigthm}, the number $k~{\rm mod}~2$ and the spin condition of $M$ are determined by the Steenrod square $Sq^2$. Since the attaching maps of the top cells of $\Sigma \mathbb{C}P^3$ and $\Sigma V_3$ are $E\pi_2$ of order $6$ and $3E\pi_2$ of order $2$ respectively by Proposition \ref{Vd=1decprop}, after localization at $3$ we can consider the Steenrod power $\mathcal{P}^1: H^3(\Sigma M;\mathbb{Z}/3\mathbb{Z})\rightarrow H^7(\Sigma M;\mathbb{Z}/3\mathbb{Z})$. Then since $\Sigma V_3\simeq_{(3)} S^3\vee S^5\vee S^7$, $\mathcal{P}^1$ acts trivially on its cohomology. In contrast, $\Sigma \mathbb{C}P^3\simeq_{(3)} S^3\cup_{\alpha_1} e^{7}\vee S^5$ withe $\alpha_1$ an element detected by $\mathcal{P}^1$ \cite[Section 1.5.5]{Har}. Hence, $\Sigma \mathbb{C}P^3$ and $\Sigma V_3$ can be distinguished by
$\mathcal{P}^1$.
Moreover, the stable cube element $\eta^3_n\in \pi_{n+3}(S^n)$ ($n\geq 2$) is detected by the secondary operation $\mathbb{T}$ \cite[Exercise 4.2.5]{Har}. And there is no indeterminacy since in the either case $S^4\cup_{\lambda\eta_4^{3}}e^8$ splits off as a wedge summand of the double suspension of the manifold. From the above discussions on cohomology operations, we can prove the corollary easily by the decompositions in Theorem \ref{M6d=1thm}.
\end{proof}
\section{Some computations on homotopy groups of odd primary Moore spaces}
\label{sec: hgps}
In this section, we work out the homotopy groups of Moore spaces used in Section \ref{sec: hdec}.
Consider the Moore space $P^{2n+1}(p^r)$ with $n\geq 1$, $p\geq 3$ and $r\geq 1$. We have the homotopy fibration
\begin{equation}\label{Fpeq}
F^{2n+1}\{p^r\}\longrightarrow P^{2n+1}(p^r)\stackrel{q}{\longrightarrow} S^{2n+1}, 
\end{equation}
where $q$ is the pinch map of the bottom cell. Cohen-Moore-Neisendorfer proved the following the famous decomposition theorem.

\begin{theorem}\cite{CMN, N}\label{CMNthm}
Let $p$ be an odd prime.
There is a $p$-local homotopy equivalence
\[
\Omega F^{2n+1}\{p^r\}\simeq_{(p)} S^{2n-1}\times \prod_{k=1}^{\infty} S^{2p^kn-1}\{p^r\}\times \Omega\Sigma \mathop{\bigvee}\limits_{\alpha} P^{n_\alpha}(p^r),
\]
where $S^i\{p^r\}$ is the homotopy fibre of the degree map $p^r: S^i\rightarrow S^i$, and $\mathop{\bigvee}\limits_{\alpha} P^{n_\alpha}(p^r)$ is an infinite bouquet of mod $p^r$ Moore spaces with only finitely many in each dimension and the least value of $n_\alpha$ is $4n-1$. ~$\qqed$
\end{theorem}
We also need the following classical result.
\begin{lemma}\cite[Proposition 6.2.2]{N2} \label{sma-wedlemma}
Let $p$ be an odd prime.
\[
 \hspace{4cm}
P^{m}(p^r)\wedge P^{n}(p^r)\simeq P^{m+n}(p^r)\vee P^{m+n-1}(p^r). 
 \hspace{4cm}\Box
\]
\end{lemma}

\begin{lemma}\cite{ST, So}\label{pinpnlemma}
Let $p$ be an odd prime.
\[
\pi_{3}(P^3(p^r))=\mathbb{Z}/p^r\mathbb{Z}, \ \ \  \pi_{n}(P^n(p^r))=0,
\]
for $n\geq 4$.
\end{lemma}
\begin{proof}
The cases when $n=3$ and $4$ were already proved in \cite[Lemma 2.1]{ST} and \cite[Lemma 3.3]{So} respectively, while the remaining cases follow immediately from the Freudenthal Suspension Theorem.
\end{proof}

\begin{lemma}\label{pin+1pnlemma}
Let $p$ be an odd prime.
\[
\pi_{n+1}(P^n(p^r))=0,
\]
for $n\geq 3$.
\end{lemma}
\begin{proof}
$\pi_4(P^3(p^r))=0$ was showed in \cite[Lemma 3.3]{So}. Let us consider $\pi_{5}(P^4(p^r))$.
By the classical EHP-sequence (Chapter XII, Theorem 2.2 of \cite{Whi}), there is the exact sequence
\[
0=\pi_4(P^3(p^r))\stackrel{}{\rightarrow} \pi_5(P^4(p^r))\stackrel{H}{\rightarrow} \pi_5(P^4(p^r)\wedge P^3(p^r)) \stackrel{P}{\rightarrow}\pi_3(P^3(p^r))\stackrel{}{\rightarrow}\pi_4(P^4(p^r))=0.
\]
By Lemma \ref{sma-wedlemma}, 
\[
\pi_5(P^4(p^r)\wedge P^3(p^r))\cong \pi_5(P^6(p^r)\vee P^7(p^r))\cong \mathbb{Z}/p^r\mathbb{Z}. 
\]
Hence, by Lemma \ref{pinpnlemma} and the above exact sequence, $P$ is an isomorphism and then $\pi_5(P^4(p^r))=0$. The remaining cases follow immediately from the Freudenthal Suspension Theorem, and this completes the proof of the lemma.
\end{proof}

In the remaining two lemmas, we exclude the case when $p=3$.
\begin{lemma}\label{pin+2pnlemma}
Let $p\geq 5$.
\[
\pi_{n+2}(P^n(p^r))=0,
\]
for $n\geq 6$.
\end{lemma}
\begin{proof}
By the Freudenthal Suspension Theorem, it suffices to show $\pi_{9}(P^7(p^r))=0$. For that let us compute $\pi_{9}(F^7\{p^r\})$ first. By Theorem \ref{CMNthm},
\[
\pi_{9}(F^7\{p^r\})\cong \pi_{8}(\Omega F^7\{p^r\})\cong \pi_{8}(S^5)_{(p)}.
\]
Since $\pi_{8}(S^5)\cong \mathbb{Z}/24\mathbb{Z}$ and $p\geq 5$, $\pi_{9}(F^7\{p^r\})=0$. Now from the exact sequence of homotopy groups of the homotopy fibration (\ref{Fpeq}) ($n=3$)
\[
0=\pi_{9}(F^7\{p^r\})\rightarrow \pi_{9}(P^7(p^r)) \rightarrow \pi_9(S^7)_{(p)}=0,
\]
we see that $\pi_{9}(P^7(p^r))=0$.
\end{proof}

\begin{lemma}\label{pi6p4pnlemma}
Let $p\geq 5$.
The suspension morphism
\[
E: \pi_{6}(P^4(p^r))\rightarrow \pi_{7}(P^5(p^r))\cong \mathbb{Z}/p^r\mathbb{Z}
\]
is trivial.
\end{lemma}
\begin{proof}
On the one hand there is the EHP-sequence of $P^4(p^r)$
\[
\pi_6(P^4(p^r))\stackrel{E}{\rightarrow}\pi_7(P^5(p^r))\stackrel{H}{\rightarrow}\pi_7(P^5(P^r)\wedge P^4(p^r))\stackrel{}{\rightarrow}\pi_5(P^4(p^r))=0,
\]
where $\pi_5(P^4(p^r))=0$ by Lemma \ref{pin+1pnlemma}, and 
\[
\pi_7(P^5(P^r)\wedge P^4(p^r))\cong \pi_7(P^8(P^r)\wedge P^9(p^r))\cong \mathbb{Z}/p^r\mathbb{Z}
\]
by Lemma \ref{sma-wedlemma}. It follows that 
\begin{equation}\label{75eq1}
\pi_7(P^5(p^r))/{\rm Im}(E)\cong \mathbb{Z}/p^r\mathbb{Z}.
\end{equation}
On the other hand there is the EHP-sequence of $P^5(p^r)$
\[
\pi_9(P^6(P^r)\wedge P^5(p^r))\stackrel{P}{\rightarrow}\pi_7(P^5(p^r))\stackrel{}{\rightarrow}\pi_8(P^6(p^r))=0,
\]
where $\pi_8(P^6(p^r))=0$ by Lemma \ref{pin+2pnlemma}, and 
\[
\pi_9(P^6(P^r)\wedge P^5(p^r))\cong \pi_9(P^{10}(P^r)\wedge P^{11}(p^r))\cong \mathbb{Z}/p^r\mathbb{Z}
\]
by Lemma \ref{sma-wedlemma}.
It follows that 
\begin{equation}\label{75eq2}
\pi_7(P^5(p^r))\cong \mathbb{Z}/p^r\mathbb{Z}/{\rm Ker}(P).
\end{equation}
Combining (\ref{75eq1}) and (\ref{75eq2}), we see that $\pi_{7}(P^5(p^r))\cong \mathbb{Z}/p^r\mathbb{Z}$, and ${\rm Im}(E)={\rm Ker}(P)=0$. The proof of the lemma is completed.
\end{proof}

\bibliographystyle{amsalpha}

\end{document}